\documentclass{amsart}
\usepackage[latin1]{inputenc}
\usepackage{ae,aecompl,amsbsy,amssymb,amsmath,amsthm,
eurosym,amsfonts,epsfig,graphicx,graphics,verbatim,enumerate,esint,color,MnSymbol}
\usepackage{color,soul}

\usepackage{url}
\definecolor{lightblue}{rgb}{.85,.93,1}
\sethlcolor{lightblue}

\newcommand{\cdotroomy}{\,\cdot\,}

\newcommand{\dmu}{\,\mathrm{d}\mu}

\newcommand{\angles}[1]{\langle #1 \rangle}

\newcommand{\br}{\mathbb{R}}

\newcommand{\cd}{\mathcal{D}}

\newcommand{\cf}{\mathcal{F}}

\newcommand{\cs}{\mathcal{S}}
\newcommand{\ct}{\mathcal{T}}
\newcommand{\be}{\begin{equation*}}
\newcommand{\ee}{\end{equation*}}
\newcommand{\ba}{\begin{eqnarray*}}
\newcommand{\ea}{\end{eqnarray*}}
\DeclareMathOperator{\supp}{supp}

\newcommand{\Q}{|Q|}

\newcommand{\qhat}{\hat{Q}}
\newcommand{\that}{\hat{T}}

\newcommand{\fhat}{\hat{F}}

\newcommand{\set}[1]{\{ #1 \}}

\newcommand{\abs}[1]{\lvert#1\rvert}

\theoremstyle{plain}
\newtheorem{theorem}{Theorem}[section]

\newtheorem{proposition}[theorem]{Proposition}

\theoremstyle{definition}

\newtheorem*{convention}{Convention}
\newtheorem{lemma}[theorem]{Lemma}

\theoremstyle{definition}

\theoremstyle{remark}
\newtheorem*{remark}{Remark}

\numberwithin{equation}{section}

\newcommand{\norm}[1]{\lVert#1\rVert}

\newcommand{\ch}{\textup{ch}}
\newcommand{\chf}{\textup{ch}_\mathcal{F}}

\newcommand{\ef}{E_{\mathcal{F}}}
\newcommand{\es}{E_{\mathcal{S}}}

\begin{document}

\date{\today}
\title[On pointwise domination and median oscillation decomposition]{Remark on dyadic pointwise domination and  median oscillation decomposition}

\author{Timo S. H\"anninen}
\address{Department of Mathematics and Statistics, University of Helsinki, P.O. Box 68, FI-00014 HELSINKI, FINLAND}
\email{timo.s.hanninen@helsinki.fi}

\thanks{The author is supported by the European Union through T. Hyt\"onen's ERC Starting Grant \lq Analytic-probabilistic methods for borderline singular integrals\rq.}

\begin{abstract}In this note, we do the following: 

a) By using Lacey's recent technique, we give an alternative proof for Conde-Alonso and Rey's domination theorem, which states that each positive dyadic operator of arbitrary complexity is pointwise dominated by a positive dyadic operator of zero complexity:
\[
\sum_{S\in\mathcal{S}} \langle f \rangle^\mu_{S^{(k)}} 1_S\lesssim (k+1) \sum_{S'\in\mathcal{S}'} \langle f \rangle^\mu_{S'} 1_{S'}.
\]

b) By following the analogue between median and mean oscillation, we extend Lerner's local median oscillation decomposition to arbitrary (possibly non-doubling) measures:
\[\lvert f-m(f,\hat{S_0})\rvert 1_{S_0}\lesssim \sum_{S\in\mathcal{S}} (\omega_\lambda(f;S)+\lvert m(f,S)-m(f,\hat{S}) \rvert )1_S.\]
This can be viewed as a median oscillation decomposition adapted to the dyadic (martingale) BMO. 
As an application of the decomposition, we give an alternative proof for the dyadic (martingale) John--Nirenberg inequality, and for Lacey's domination theorem, which states that each martingale transform is pointwise dominated by a positive dyadic operator of complexity zero.
\end{abstract}

\maketitle
\tableofcontents

\section*{Notation}

\begin{tabular*}{\textwidth}{c l}
\\
$\mu$ & An arbitrary locally finite Borel measure on $\br^d$.\\
$f$ & An arbitrary measurable function $f:\br^d\to \br$.\\
$k$ & An arbitrary non-negative integer. \\
$\angles{f}_Q$ & The average of $f$ on $Q$, $\angles{f}_Q:=\angles{f}^\mu_Q:=\frac{1}{\mu(Q)}\int_Q f \dmu$.\\
$L^p$ & $L^p:=L^p(\mu)$.\\
$\cd$ & The collection of dyadic cubes.\\
$\hat{Q}$ & The dyadic parent of a dyadic cube $Q$\\
$Q^{(k)}$ & The $k$th dyadic ancestor of a dyadic cube $Q$, \\
& defined recursively by $Q^{(k+1)}=\widehat{Q^{(k)}}.$ \\
$\ch_\cd^{(k)}(Q)$& The $k$th dyadic descendants of a dyadic cube $Q$,\\
& defined by $\ch_\cd^{(k)}(Q):=\{Q'\in\cd : Q'^{(k)}=Q\}$.\\
$\chf(F)$ & The {\it $\cf$-children} $\chf(F)$ of a dyadic cube $F$,    \\
& defined by $\chf(F):=\{F'\in\cd: \text{$F'$ maximal such that $F'\subsetneq F$}\}$.\\
$\ef(F)$ & $\ef(F):=F\setminus \bigcup_{F'\in\chf(F)} F'$.\\
$\big(g(x)\big)_{x\in Q}$ & The notation for the constant value of a function $g$ on $Q$.  \\
& It is implicitly understood that the  \\
& function $g:\br^d\to \br$ is constant  on $Q$. \\
$m(f,Q)$ & Any median of $f$ on $Q$, defined in Subsection \ref{sec_definitions}.\\
$r_\lambda(f,Q)$ & The relative median oscillation of $f$ (about zero) on $Q$, \\
& defined in Subsection \ref{sec_definitions}.\\
& It is implicitly understood that $\lambda\in (0,1/2)$.\\
$\omega_\lambda(f,Q)$ & The median oscillation of $f$ on $Q$, defined in Subsection \ref{sec_definitions}. \\
&It is implicitly understood that $\lambda\in (0,1/2)$.\\ 
&\\
\multicolumn{2}{ l }{$\bullet $ A collection $\cf\subseteq \cd$ is {\it sparse} if there exists $\gamma \in(0,1)$    }\\
\multicolumn{2}{ l }{such that $\sum_{F'\in\chf(F)} \mu(F')\leq \gamma \mu(F)$ for every $F\in\cf$.}\\
\end{tabular*}
\section{Introduction}
In this note, by adapting Lacey's recent technique \cite{lacey2015}, we give an alternative proof for Conde-Alonso and Rey's domination theorem \cite{condealonso2014}. Furthermore, we extend Lerner's local median oscillation decomposition \cite{lerner2010,lerner2013}  to arbitrary (possibly non-doubling) measures.

First, we consider the domination theorem. Conde-Alonso and Rey proved that:
\begin{theorem}[Pointwise domination theorem for positive dyadic operators, Theorem A in  \cite{condealonso2014}]\label{thm_domination} Let $\cs$ be a sparse collection that contains a maximal cube. Then there exists a sparse collection $\ct$ such that
$$
\sum_{S\in\cs} \angles{f}_{S^{(k)}} 1_S\lesssim (k+1) \sum_{T\in\ct} \angles{f}_{T} 1_T$$
$\mu$-almost everywhere.
The collection $\ct$ depends on the measure $\mu$, the collection $\cs$, the integer $k$, and the function $f$.
\end{theorem} 
\begin{remark}This result improves on Lerner's domination result \cite[Proof of Theorem 1.1]{lerner2013}, which states the domination in any Banach function space norm (in particular, in the $L^p$ norm). 
\end{remark}
In Section \ref{sec_pointwisedomination}, we give an alternative proof for Theorem \ref{thm_domination} by adapting Lacey's recent technique \cite[Proof of Theorem 2.4]{lacey2015}.

Then, we consider the local median oscillation decomposition. Lerner proved that:
\begin{theorem}[Median oscillation decomposition, Theorem 1.1 in \cite{lerner2010} and Theorem 4.5 in \cite{lerner2013}]\label{thm_lerner}Let $\mu$ be a locally finite Borel measure. Assume that $\mu$ is doubling. Let $F_0$ be an initial cube. Then, there exists a sparse collection $\cf$ of dyadic subcubes of $F_0$ such that
$$
\abs{f-m(f,F_0)}1_{F_0}\lesssim \sum_{F\in\cf} \omega_\lambda(f;F) 1_F
$$
$\mu$-almost everywhere. The collection $\cf$ depends on the initial cube $F_0$ and the function $f$, and the parameter $\lambda$ depends on the doubling constant.
\end{theorem} 
\begin{remark}The original decomposition by Lerner in \cite[Theorem 1.1]{lerner2010} and \cite[Theorem 4.5]{lerner2013} contains an additional term (a median oscillation maximal function), which was removed by Hyt\"onen in \cite[Theorem 2.3]{hytonen2014}. Furthermore, the localization on an initial cube was removed by Lerner and Nazarov \cite[Theorem 10.2]{lerner2014}.
\end{remark}
In Section \ref{sec_mediandecomposition}, we extend Theorem \ref{thm_decompositiongeneralmeasures} as follows:
\begin{theorem}[Median oscillation decomposition, adapted to the dyadic martingale BMO]\label{thm_decompositiongeneralmeasures}Let $\mu$ be an arbitrary (possibly non-doubling) locally finite Borel measure. Let $F_0$ be an initial cube. Then, there exists a sparse collection $\cf$ of dyadic subcubes of $F_0$ such that
$$
\abs{f-m(f,\hat{F_0})}1_{F_0}\lesssim \sum_{F\in\cf} (\omega_\lambda(f;F)+\abs{m(f,F)-m(f,\hat{F})} )1_F
$$
$\mu$-almost everywhere. The collection $\cf$ depends on the initial cube $F_0$ and the function $f$. 
\end{theorem} 
\begin{remark}Because of the analogy between median oscillation and mean oscillation, this can be viewed as a median oscillation decomposition adapted to the dyadic (martingale) BMO, as explained in Subsection \ref{discussion_bmo}.
\end{remark}

To keep this note as short as possible, only a tiny part of the story on the dyadic positive operators (story which revolves around the $A_2$ theorem) is told; For a bigger picture, see, for example, the introduction and the discussion in Lacey's paper \cite{lacey2015}, or Hyt\"onen's survey on the $A_2$ theorem \cite{hytonen2014}.

\section{Pointwise domination theorem for positive dyadic operators}\label{sec_pointwisedomination}
\subsection{Alternative proof by adapting Lacey's recent technique}
\begin{proof}[Alternative proof for Theorem \ref{thm_domination}]To avoid writing the absolute value $\abs{\cdotroomy}$, we assume that the function $f$ is non-negative. 

We define
\begin{equation}\label{def_akoperator}
A_kf:=\sum_{S\in\cs} \angles{f}_{S^{(k)}} 1_S =\sum_{Q\in\cd} \angles{f}_Q \sum_{\substack{S\in\cs :\\ S^{(k)}=Q}}1_S=:\sum_{Q\in\cd} \angles{f}_Q \eta_Q.
\end{equation} We observe that each auxiliary function $n_Q$ satisfies $\eta_{Q}\leq 1_Q$. Moreover, the auxiliary function $\eta_{Q}$ is constant on each $Q'\in\ch_\cd^{(k)}(Q)$. 

For each $F\in\cd$, let $\chf(F)$ denote the collection of all the maximal $F'\in\{F'\in \cd : F'\subseteq F\}$ such that
\begin{equation}\label{eq_stop_weak}
\Big(\sum_{\substack{Q\in\cd :\\ F'^{(k)}\subseteq Q \subseteq F}} \angles{f}_Q \eta_Q(x)\Big)_{x\in F'} > 4 \norm{A_k}_{L^1\to L^{1,\infty}}\angles{f}_F,
\end{equation}
or
\begin{equation}\label{eq_stop_principal}
\angles{f}_{F'}> 4 \angles{f}_F.
\end{equation}
We observe that the weak-$L^1$ estimate implies that the cubes $F'$ satisfying the first stopping condition satisfy the measure condition: 
\begin{equation*}
\begin{split}
\sum_{F'} \mu(F')&=\sum_{F'} \mu(F'\cap \{\sum_{\substack{Q\in\cd :\\ F'^{(k)}\subseteq Q \subseteq F}} \angles{f}_Q \eta_Q > 4 \angles{f}_F \})\\
&\leq \mu(\{A_k(1_Ff) > 4 \norm{A_k}_{L^1\to L^{1,\infty}}\angles{f}_F \})\leq \frac{1}{4} \mu(F).
\end{split}
\end{equation*}
Similarly, the cubes $F'$ satisfying the second stopping condition satisfy the measure condition $\sum_{F'} \mu(F')\leq \frac{1}{4}\mu(F)$. Altogether, $\sum_{F'\in\ch(F)} \mu(F')\leq \frac{1}{2} \mu(F).$

Now, by decomposing the summation and invoking the stopping conditions,
\begin{equation}\label{eq_iterationstep}
\begin{split}
S_{\subseteq F}:=&\sum_{Q\in\cd: Q\subseteq F} \angles{f}_Q \eta_Q\\
=&\sum_{Q\in\cd: Q\subseteq F} \angles{f}_Q \eta_Q1_{\ef(F)}+\sum_{F'\in\chf(F)}\sum_{Q\in\cd: Q\subseteq F} \angles{f}_Q \eta_Q1_{F'} \\
=&\sum_{Q\in\cd: Q\subseteq F} \angles{f}_Q \eta_Q1_{\ef(F)}+ \sum_{F'\in\chf(F)}\sum_{\substack{Q\in\cd: \\ F'^{(k+1)}\subseteq Q \subseteq F}} \angles{f}_Q \eta_Q1_{F'} \\
&+\sum_{F'\in\chf(F)}\sum_{\substack{Q\in\cd: \\F'^{(1)}\subseteq Q\subseteq \min\{F'^{(k)},F\}}} \angles{f}_Q \eta_Q1_{F'}+ \sum_{F'\in\chf(F)} \sum_{Q\in\cd : Q\subseteq F'}\angles{f}_Q \eta_Q \\
\leq& 4 \norm{A_k}_{L^1\to L^{1,\infty}}\angles{f}_F (1_{\ef(F)}+  \sum_{F'\in\chf(F)}1_{F'})+k 4\angles{f}_F 1_F+\sum_{F'} S_{\subseteq F'},
\end{split}
\end{equation}
where the last step follows from the following observations:
\begin{itemize}
\item 
For each $x\in \ef(F)$, every $R\in\{R\in\cd : R\subseteq F\}$ such that $R\ni x$ satisfies the opposite of the stopping condition \eqref{eq_stop_weak}. Therefore,
$$
\sum_{Q\in\cd: Q\subseteq F} \angles{f}_Q \eta_Q(x)=\lim_{\substack{Q\in\cd: R\ni x, \ell(R)\to 0}} \Big(\sum_{Q\in\cd: R^{(k)}\subseteq Q\subseteq F} \angles{f}_Q\eta_Q\Big)_{x\in R} \leq  4 \norm{A_k}_{L^1\to L^{1,\infty}}\angles{f}_F.
$$
\item 
By maximality, the cube $F'^{(1)}$ satisfies the opposite of the stopping condition \eqref{eq_stop_weak}. Therefore,
$$
\sum_{\substack{Q\in\cd: \\ F^{(k+1)}=(F'^{(1)})^{(k)}\subseteq Q \subseteq F}} \angles{f}_Q \eta_Q1_{F'} \leq 4 \norm{A_k}_{L^1\to L^{1,\infty}}\angles{f}_F.
$$
\item 
By maximality, every cube $Q\in\cd$ such that $F'^{(1)}\subseteq Q\subseteq \min\{F'^{(k)},F\}$ satisfies the opposite of the stopping condition \eqref{eq_stop_principal}. Therefore, $\angles{f}_Q\leq 4 \angles{f}_F$ for all such cubes $Q$. 
\end{itemize}
By Proposition \ref{prop_weaklone}, we have $\norm{A_k}_{L^1(\mu)\to L^{1,\infty}(\mu)}\lesssim 1$. Note that the weak $L^1$ estimate for the operator $A_k$ is independent of $k$, whereas the weak $L^1$ estimate for the adjoint operator $A_k^*$ depends linearly on $k$.  The proof is completed by iteration, starting from the maximal cube (which exists, by assumption) of the collection $\cs$.
\end{proof}

\subsection{Weak-$L^1$ estimate for positive dyadic operators}
\begin{proposition}[Weak $L^1$ for positive dyadic operators]\label{prop_weaklone}Let $\mu$ be a locally finite Borel measure. Let $A_k$ be defined as in \eqref{def_akoperator}. Then
$$
\norm{A_kf}_{L^1\to L^{1,\infty}}\lesssim 1.
$$
\end{proposition}
\begin{remark}
The weak $L^1$ estimate for the operator $A_k$ is proven using the Calder\'on--Zygmund decomposition: In the case of a doubling measure, this is proven as in \cite[Proof of Proposition 5.1]{hytonen2012} or in \cite[Proof of Lemma 5.4]{lerner2013}; In the case of an arbitrary (possibly non-doubling) measure, the Calder\'on--Zygmund decomposition contains an additional term, for which the weak $L^1$ estimate is checked in what follows. 
\end{remark}
We prove the weak-$L^1$ boundedness by using the Calder\'on--Zygmund decomposition for general measures obtained by L\'opez--S\'anchez, Martell, and Parcet:

\begin{lemma}[Calder\'on--Zygmund decomposition for general measures, Theorem 2.1 in \cite{lopez2012}]\label{lem_calderon}Let $\mu$ be a locally finite Borel measure on $\br^d$. Assume that the measure of each $d$-dimensional quadrant is infinite. Then, for each $f\in L^1$ and $\lambda>0$, there exists a decomposition
$$
f=g+b+\beta
$$
such that the pieces satisfy the following properties:
\begin{itemize}
\item The function $g$ satisfies
$$
\norm{g}^p_{L^p}\lesssim_p \lambda^{p-1} \norm{f}_{L^1}
$$
for every $1\leq p <\infty$.

\item The function $b$ has the decomposition $b=\sum_{T\in\ct} b_T$ such that
$$
\supp(b_T)\subseteq T, \quad \int b_T \dmu=0,\quad   \sum_{T\in\ct} \norm{b_T}_{L^1}\lesssim \norm{f}_{L^1}.
$$
\item The function $\beta$ has the decomposition $\beta=\sum_{T\in\ct} \beta_{\that}$ such that
$$
\supp(\beta_{\that})\subseteq \that, \quad \int \beta_{\that} \dmu=0,\quad \sum_{T\in\ct} \norm{b_{\that}}_{L^1}\lesssim \norm{f}_{L^1},
$$
and $\beta_{\that}$ is constant on $T$ and on $\that\setminus T$.
\item The cubes $T$ are the maximal (which exist because, by assumption, the measure of each $d$-dimensional quadrant is infinite) dyadic cubes such that $\angles{\abs{f}}_T>\lambda$ . Hence, they are pairwise disjoint, and their union $\Omega:=\bigcup_{T} T$ satisfies
$
\mu(\Omega)\leq \frac{1}{\lambda} \int \abs{f} \dmu.
$
\end{itemize}
\end{lemma}
\begin{proof}[Proof of Proposition \ref{prop_weaklone}]We suppress the complexity $k$ in the notation. By using the Calder\'on--Zygmund decomposition (Lemma \ref{lem_calderon}), we decompose
$$
\mu(\set{\abs{Af}>\lambda})\leq \mu(\set{\abs{Ag}>\frac{\lambda}{3}})+\mu(\set{\abs{Ab}>\frac{\lambda}{3}}\cap \Omega^c)+\mu(\set{\abs{A\beta}>\frac{\lambda}{3}}\cap \Omega^c)+\mu(\Omega).
$$
By Chebyshev's inequality together with the $L^p\to L^p$ boundedness of the operator $A_k$, we have
$$
\mu(\set{\abs{Ag}>\frac{\lambda}{3}})\lesssim \frac{1}{\lambda^p} \norm{Ag}_{L^p}\lesssim_p \frac{1}{\lambda^p} \norm{g}_{L^p}^p \lesssim_p  \frac{1}{\lambda^p} \lambda^{(p-1)} \norm{f}_{L^1}.
$$
We observe that  $1_{T^c} A(h_T)=0$ whenever $h_T$ is such that $\supp(h_T)\subseteq T$ and $\int h_T \dmu=0$. This together with Chebyshev's inequality implies that 
$$
\mu(\set{Ab>\frac{\lambda}{3}}\cap \Omega^c)\lesssim \frac{1}{\lambda} \sum_{T} \int_{T^c} \abs{Ab_T} \dmu=0,
$$
and
\begin{equation*}
\begin{split}
\mu(\set{A\beta >\frac{\lambda}{3}}\cap \Omega^c)&\lesssim \frac{1}{\lambda} \sum_{T} \int_{T^c} \abs{A\beta_{\that}} \dmu= \frac{1}{\lambda}\sum_{T} \int_{\that\setminus T} \abs{A(\beta_{\that})}\dmu.
\end{split}
\end{equation*}
Since $\beta_{\that}$ is contant on $\that\setminus T$, we have $$1_{\that\setminus T} A(\beta_{\that})=\sum_{Q:Q\subseteq \that\setminus T} \angles{\beta_{\that}}_{Q} \eta_Q=\angles{\beta_{\that}}_{\that\setminus T} \sum_{Q:Q\subseteq \that\setminus T} \eta_Q.$$
Recall that, by definition, $\eta_Q:=\sum_{S\in\cs: S^{(k)}=Q}1_S$, where $\cs$ is a sparse collection. Therefore, by sparseness,
\begin{equation*}
\begin{split}
&\int_{\that\setminus T} \abs{A(\beta_{\that})}\dmu\leq \abs{\angles{\beta_{\that}}_{\that\setminus T}}  \sum_{S\in\cs:S\subseteq \that\setminus T}\mu(S)\\
&\lesssim \abs{\angles{\beta_{\that}}_{\that\setminus T}}  \sum_{S\in\cs:S\subseteq \that\setminus T}\mu(\es(S)) \leq \abs{\angles{\beta_{\that}}_{\that\setminus T}} \mu(\that\setminus T)\leq \norm{\beta_{\that}}_{L^1}.
\end{split}
\end{equation*}
The proof is completed by the property $\sum_{T\in\ct} \norm{b_{\that}}_{L^1}\lesssim \norm{f}_{L^1}$.

\end{proof}

\section{Median oscillation decomposition}\label{sec_mediandecomposition}
\begin{convention}Throughout this section, the parameter $\lambda$ is an arbitrary real number such that $0<\lambda<1/2$.\end{convention}

\subsection{Definition of median and median oscillation}\label{sec_definitions}

\begin{itemize}
\item The {\it median} $m(f;Q)$ of a function $f$ on a cube $Q$ is defined as any real number such that 
$$
\frac{\mu(Q\cap \set{f>m(f;Q)})}{\mu(Q)}\leq \frac{1}{2}\quad\text{ and } \quad \frac{\mu(Q\cap \set{f<m(f;Q)})}{\mu(Q)}\leq \frac{1}{2}.
$$
\item The {\it relative median oscillation} $r_\lambda(f;Q)$ of a function $f$ ({\it about zero}) on a cube $Q$ is defined by
$$
r_\lambda(f;Q):=\min\set{r\geq 0 : \mu(Q\cap \set{\abs{f}>r})\leq \lambda \mu(Q)}.
$$
Note that, by means of decreasing rearrangement, the relative median oscillation is written as $r_\lambda(f;Q)=(1_Qf)^*(\lambda \mu(Q))$.  The quantity $r_\lambda(f-c;Q)$ is the {\it   relative median oscillation } of a function $f$ {\it about a real number $c$} on a cube $Q$.  
\item The {\it median oscillation} $\omega_\lambda(f;Q)$ of a function $f$ on a cube $Q$ is defined by
$$
\omega_\lambda(f;Q):=\inf_{c\in\br} r_\lambda(f-c;Q).
$$
\end{itemize}

\subsection{Properties of median and median oscillation}
For reader's convenience, we summarize the properties of median that we need. The properties are all well-known. For proofs, see, for example, the lecture notes \cite[Section 5]{hytonen2014b}. 
\begin{lemma}[Every median quasiminimizes the median oscillation]\label{lemma_quasiminimizer}We have $$
r_\lambda(f-m(f;Q);Q) \leq 2 \omega_\lambda(f;Q).$$
\end{lemma}
\begin{lemma}[Median is linear]We have $$m(f+c;Q)=m(f;Q)+c.$$ Since median is not unique, this slight abuse of notation is understood as an identity for the set of all medians: $\{m: \text{$m$ is a median of $(f+c)$ on $Q$}\}=\{m': \text{$m'$ is a median of $f$ on $Q$}\}+c$.
\end{lemma}
\begin{lemma}[Median is controlled by the relative median oscillation]\label{lemma_3rproperty}We have $$
\abs{m(f;Q)-c}\leq 3 r_\lambda(f-c;Q).$$
\end{lemma}
\begin{proof}
Using the fact that every median quasiminimizes the median oscillation (Lemma \ref{lemma_quasiminimizer}), and the definition of median oscillation, we have
$$r_\lambda(f-m(f;Q))\leq 2 \omega_\lambda(f;Q)\leq 2r_\lambda(f-c;Q).$$
This, by the definition of relative median oscillation, implies that
\begin{equation*}
\begin{split}
\mu(\set{\abs{f-m(f;Q)}>2r_\lambda(f-c;Q)})&\leq \lambda\mu (Q), \text{and }\mu(\set{\abs{f-c}>r_\lambda(f-c;Q)})\leq \lambda \mu(Q).
\end{split}
\end{equation*}
From this together with the implicit assumption $0<\lambda<1/2$, it follows that there exists $x\in Q$ such that $\abs{f(x)-m(f;Q)}\leq 2 r_\lambda(f-c;Q)$ and $\abs{f(x)-c}\leq r_\lambda(f-c;Q)$. The proof is completed by the triangle inequality.
\end{proof}
\begin{lemma}[Fujii's Lemma]\label{lemma_fujii}We have
$$
\lim_{\substack{Q\in\cd:\\Q \ni x, \ell(Q)\to 0 }} m(f;Q)=f(x)
$$
for $\mu$-almost every $x\in\br^d$.
\end{lemma}
\begin{lemma}[Relative median oscillation is controlled by the weak $L^1$ norm]\label{lemma_weakl1control}We have
$$
r_\lambda(f;Q)\leq \frac{1}{\lambda} \frac{\norm{f}_{L^{1,\infty}}}{\mu(Q)}.
$$
\end{lemma}

\subsection{Proof of the decomposition adapted to the dyadic BMO}

\begin{proof}[Proof of Theorem \ref{thm_decompositiongeneralmeasures}]From Lemma \ref{lemma_3rproperty} and Lemma \ref{lemma_quasiminimizer}, it follows that
\begin{equation}\label{comparisionofoscillationquantities}
\begin{split}
\omega_\lambda(f;F)+\abs{m(f,F)-m(f,\hat{F})} &\eqsim r_\lambda(f-m(f,F))+\abs{m(f,F)-m(f,\hat{F})}\\
&\eqsim r_\lambda(f-m(f;\fhat);F).
\end{split}
\end{equation}

For each $F\in\cd$, let $\chf(F)$ denote the collection of all the maximal $F'\in\{F'\in \cd : F'\subseteq F\}$ such that
\begin{equation}\label{eq_stop_median}
\abs{m(f;F')-m(f;\hat{F})}>3r_\lambda(f-m(f;\fhat);F)
\end{equation}
By decomposing and using the stopping condition,
\begin{equation*}
\begin{split}
\abs{f-m(f,\fhat)}1_{F}&\leq \abs{f-m(f;\fhat)}1_{\ef(F)}+ \sum_{F'\in\chf(F)} \abs{m(f;\hat{F'})-m(f;\hat{F})}1_{F'}+ \\
&\phantom{=}\sum_{F'\in\chf(F)} \abs{f-m(f,\hat{F'})}1_{F'}\\
&\leq  3 r_\lambda(f-m(f,\fhat);F)1_F+ \sum_{F'\in\chf(F)} \abs{f-m(f,\hat{F'})}1_{F'},
\end{split}
\end{equation*}
where the last step follows from the following observations:
\begin{itemize}
\item
For each $x\in \ef(F)$, every cube $Q\in\{Q\in\cd : Q\subseteq F\}$ such that $Q\ni x$ satisfies the opposite of the stopping inequality \eqref{eq_stop_median}. Therefore, by Fujii's Lemma (Lemma \ref{lemma_fujii}),
$$
\abs{f(x)-m(f;\fhat)}=\lim_{\substack{Q\in\cd:\\Q \ni x, \ell(Q)\to 0 }} \abs{m(f;Q)-m(f;\fhat)}\leq 3 r_\lambda(f-m(f,\fhat);F)
$$ 
$\mu$-almost every $x\in \ef(F)$.
\item 
By maximality, the cube $\hat{F'}$ satisfies the opposite of the stopping inequality \eqref{eq_stop_median}. Therefore,
$$
\abs{m(f;\hat{F'})-m(f;\hat{F})}\leq 3 r_\lambda(f-m(f,\fhat);F).
$$
\end{itemize}

Finally, we check that $\sum_{F'\in\chf(F)}\mu(F')\leq 2\lambda \mu(F)$. Let $\kappa \in (0,1/2)$ be an auxiliary parameter. We note the following assertion: 
$$\text{If $\mu(Q\cap \set{\abs{f-c}>r})\leq \kappa \mu(Q),$ then $\abs{m(f;Q)-c}\leq 3 r$;}$$ This is because $\mu(Q\cap \set{\abs{f-c}>r})\leq \kappa \mu(Q)$ implies, by definition, that $r_\kappa(f-c;Q)\leq r$, from which, by Lemma \ref{lemma_3rproperty}, it follows  that $\abs{m(f;Q)-c}\leq 3 r$. The contrapositive of this assertion applied to the stopping inequality \eqref{eq_stop_median} (where we have $Q:=F'$, $c:=m(f,\fhat)$ and $r:=r_\lambda(f-m(f;\fhat);F)$) implies that
\begin{equation}
\label{eq_ontheonehand}
\mu(F'\cap \set{\abs{f-m(f,\fhat)}>r_\lambda(f-m(f;\fhat);F)})> \kappa \mu(F').
\end{equation}
On the other hand, by definition,
\begin{equation}
\label{eq_ontheotherhand}
\lambda \mu(F) \geq \mu(F\cap \set{\abs{f-m(f;\hat{F})}>r_\lambda(f-m(f;\fhat);F)}).
\end{equation}
Summing over the  cubes $F'$ (which are pairwise disjoint and satisfy $F'\subseteq F$)  in the inequality \eqref{eq_ontheonehand},  combining this with the inequality \eqref{eq_ontheotherhand}, and taking $\kappa\to 1/2$ yields
$$
\sum_{F'\in\chf(F)}\mu(F') \leq 2 \lambda \mu(F).
$$
The proof is completed by iteration.

\end{proof}
\subsection{Corollaries}
The dyadic (martingale) BMO norm is defined by
$$
\norm{f}_{BMO(\mu)}:=\sup_{Q\in\cd} \frac{1}{\mu(Q)}\int_Q \abs{f-\angles{f}_{\qhat}} \dmu.
$$
Note that, whenever the measure $\mu$ is doubling, the dyadic (martingale) BMO norm is comparable to the usual BMO norm: $\norm{f}_{BMO(\mu)}\eqsim_{\mu} \sup_{Q\in\cd} \frac{1}{\mu(Q)}\int_Q \abs{f-\angles{f}_Q} \dmu$.
\begin{proposition}[John--Nirenberg]Let $\mu$ be a locally finite Borel measure. Then, there exist positive constants $c$ and $C$ such that
$$
\frac{1}{\mu(Q)}\int_Q \exp(c \abs{f-\angles{f}_{\qhat}}/\norm{f}_{\text{BMO}})\dmu\leq C
$$
for every $f\in\text{BMO}$.
\end{proposition}
\begin{proof}[Proof by the dyadic median oscillation decomposition]By using the inequalities
$$
r_\lambda(f;Q)\lesssim_\lambda \frac{1}{\mu(Q)}\int_Q \abs{f} \dmu\quad\text{and} \quad \abs{m(f;Q)}\leq 3 r_\lambda(f;Q),
$$
of which the first follows from Chebyshev's inequality and the second is stated in Lemma \ref{lemma_3rproperty}, and by using the linearity of median, we obtain
$$
r_\lambda(f-m(f;\qhat);Q)\lesssim \norm{f}_{BMO}.
$$
By the median oscillation decomposition (Theorem \ref{thm_decompositiongeneralmeasures}), there exists a sparse collection $\cs$ of dyadic subcubes of $Q$ such that 
$$
\abs{f-m(f,\qhat)}1_Q\lesssim \sum_{S\in\cs} r_\lambda(f-m(f;\hat{S});S) 1_S.
$$
Altogether,
$$ \abs{f-\angles{f}_{\qhat}}1_Q \leq \abs{f-m(f,\qhat)}1_Q+  \abs{m(f;\qhat)-\angles{f}_{\qhat}}1_Q \lesssim \norm{f}_{BMO}\sum_{S\in\cs} 1_{S}.
$$
By sparseness, $\mu(\{\sum_{S\in\cs} 1_{S}=k\})\lesssim 2^{-k}\mu(Q)$, from which the exponential integrability follows by splitting the integration as $\int_{Q}=\sum_{k=0}^\infty \int_{\{\sum_{S\in\cs} 1_{S}=k\}}$.
\end{proof}
The {\it martingale transform} $T$ associated with the (constant) coefficients $\epsilon_Q$ satisfying $\abs{\epsilon_Q}\leq 1$ is defined by 
$$Tf:=\sum_{\substack{Q\in\cd}} \epsilon_Q D_Qf:=\sum_{\substack{Q\in\cd}} \epsilon_Q \big(\sum_{Q'\in\ch_\cd(Q)} \angles{f}_{Q'}1_{Q'}-\angles{f}_Q\big). $$
Lacey \cite[Theorem 2.4]{lacey2015} proves that each martingale transform is pointwise dominated by a positive dyadic operator of zero complexity. Alternative proof for this is as follows: First, use the median oscillation decomposition (Theorem \ref{thm_decompositiongeneralmeasures}) to yield the domination by positive dyadic operators of complexity zero and one. Then, apply the domination for positive dyadic operators (Theorem \ref{thm_domination}) to reduce the complexity to zero.
\begin{proposition}[A pointwise domination theorem for martingale transforms, see Lacey's Theorem 2.4 in \cite{lacey2015} for a stronger version]Let $F_0$ be an initial cube. Assume that $f:\br^d\to\br$ is a locally integrable function that is supported on the cube $F_0$. Then, there exists a sparse collection $\cf$ of dyadic subcubes of $F_0$ such that
$$
\abs{Tf}1_{F_0}\lesssim (\norm{T}_{L^1\to L^{1,\infty}}+1) \Big(\sum_{F\in\cf} \angles{\abs{f}}_F 1_F +\sum_{F\in\cf} \ \angles{\abs{f}}_{\hat{F}}1_F\Big).
$$
\end{proposition}
\begin{proof}[Proof by the median oscillation decomposition]
The theorem follows from the median oscillation decomposition (Theorem \ref{thm_decompositiongeneralmeasures}) together with an estimate for the oscillation quantities (Lemma \ref{lemma_oscillationestimate}).
\end{proof}
\begin{lemma}[Oscillations of a martingale transform]\label{lemma_oscillationestimate}Let $T$ be a martingale transform. Let $R$ be a dyadic cube. Then
$$
r_\lambda(Tf-m(Tf;\hat{R});R)\leq (\norm{T}_{ L^1\to L^{1,\infty} } +1)(\angles{\abs{f}}_{R}+\angles{\abs{f}}_{\hat{R}}),
$$
and
$$
m(Tf;R)\lesssim \norm{T}_{L^1\to L^{1,\infty}} \frac{1}{\mu(Q)}\int_{\br^d} \abs{f} \dmu
$$
\end{lemma}
\begin{proof}
Let $R$ be a dyadic cube. We split
$
1_RTf=1_RT(1_Rf)+1_RT(1_{R^c} f).
$ We observe that $1_RT(1_{R^c} f)=1_R \sum_{Q:Q \supseteq \hat{R}} \epsilon_Q D_Qf$ is constant on $R$, and denote this constant value by $c_R:= \Big( \sum_{Q\supseteq \hat{R}} \epsilon_Q D_Qf\Big)_{x\in R}$. By using the linearity of median, we write
\begin{equation*}
\begin{split}
1_R T(f)-1_R m(Tf;\hat{R})&= 1_R T(1_Rf) +1_R(c_R -c_{\hat{R}}) - 1_R m(T(1_{\hat{R}}f);\hat{R})\\
&= 1_R T(1_Rf) +1_R\epsilon_{\hat{R}} D_{\hat{R}}f - 1_R m(T(1_{\hat{R}}f);\hat{R})\\
&= 1_R T(1_Rf) +1_R \epsilon_{\hat{R}} \angles{f}_{R}-1_R \epsilon_{\hat{R}} \angles{f}_{\hat{R}} +m(T(1_{\hat{R}}f);\hat{R}).
\end{split}
\end{equation*}
Therefore,
$$
r_\lambda(Tf-m(Tf;\hat{R});R)\leq r_\lambda(T(1_Rf);R)+ \abs{m(T(1_{\hat{R}}f);\hat{R})}+ \angles{\abs{f}}_{R}+\angles{\abs{f}}_{\hat{R}}.
$$
By using the estimate $\abs{m(f;Q)}\leq 3r_\lambda(f;Q)$ (Lemma \ref{lemma_3rproperty}), and by dominating the median oscillation $r_\lambda(f;Q)$ by the weak $L^1$ estimate (Lemma \ref{lemma_weakl1control}), we obtain
\begin{equation*}
\begin{split}
r_\lambda(Tf-m(Tf;\hat{R});R)&\lesssim_\lambda \frac{\norm{T(1_Rf)}_{L^{1,\infty}}}{\mu(R)}+ \frac{\norm{T(1_{\hat{R}}f)}_{L^{1,\infty}}}{\mu(\hat{R})}+\angles{\abs{f}}_{R}+\angles{\abs{f}}_{\hat{R}}\\
&\leq  (\norm{T}_{L^1\to L^{1,\infty}} +1)(\angles{\abs{f}}_{R}+\angles{\abs{f}}_{\hat{R}}).
\qedhere\end{split}
\end{equation*}
\end{proof}

\subsection{Median oscillation decomposition adapted to the RBMO?}\label{discussion_bmo}
In the light of the analogue between median and mean, and median oscillation and mean oscillation,
$$
m(f;Q)\leftrightarrow \angles{f}_Q,\quad r_\lambda(f-c,Q)\leftrightarrow \frac{1}{\mu(Q)}\int_Q \abs{f-c}\dmu,$$
the passage from the usual BMO norm to the dyadic (martingale) BMO norm,
$$
\frac{1}{\mu(Q)}\int_Q \abs{f-\angles{f}_Q} \rightarrow \frac{1}{\mu(Q)}\int_Q \abs{f-\angles{f}_Q}\dmu+ \abs{\angles{f}_Q-\angles{f}_{\qhat}},
$$
is analogous to the passage
$$
r_\lambda(f-m(f,Q);Q)\rightarrow r_\lambda(f-m(f,Q);Q)+\abs{m(f;Q)-m(f;\qhat)},
$$
which we use to extend Lerner's local oscillation decomposition. Thus, our extension can be viewed as a local oscillation decomposition adapted to the dyadic (martingale) BMO. 

The author believes that, in the same spirit, Lerner's  local oscillation decomposition can be adapted to the RBMO space,  and that this adapted decomposition can be used to pointwise dominate non-homogeneous Calder\'on--Zygmund operators  by suitable positive averaging operators. (For the RBMO space, see \cite{tolsa2001}, and, for non-homogeneous Calder\'on--Zygmund operators, see \cite{nazarov2003}.)

We remark that a pointwise domination for non-homogeneous Calder\'on--Zygmund operators by positive averaging operators was obtained by Treil and Volberg, by adapting Lacey's technique \cite[Proof of Theorem 5.2]{lacey2015}. This result is announced by Lacey \cite[Section 6]{lacey2015}.

\bibliographystyle{plain}
\bibliography{remark_bibliography}
\end{document}